\documentclass{article}

\usepackage{amsmath,amsthm}
\usepackage[mathscr]{eucal}
\usepackage{pinlabel}
\usepackage{graphicx}

\theoremstyle{plain}
\newtheorem{tetel}{Theorem}%[section]
\newtheorem{all}[tetel]{Proposition}
\newtheorem{lemma}[tetel]{Lemma}
\newtheorem{kov}[tetel]{Corollary}

\theoremstyle{definition}\newtheorem{Def}[tetel]{Definition}
\theoremstyle{remark}

\newcommand*{\Z}{\ensuremath{\mathbf Z}}
\newcommand*{\R}{\ensuremath{\mathbf R}}
\newcommand*{\di}{\ensuremath{\mathrm d}}

\begin{document}

\title{Inner products on the Hecke algebra of the braid group}
\author{Tam\'as K\'alm\'an\\ The University of Tokyo}
\maketitle

\begin{abstract}
We point out that the Homfly polynomial (that is to say, Ocneanu's trace functional) contains two poly\-nomial-valued inner products on the Hecke algebra representation of Artin's braid group. These bear a close connection to the Morton--Franks--Williams inequality. In these structures, the sets of positive, respectively negative permutation braids become orthonormal bases. In the second case, many inner products can be geometrically interpreted through Legendrian fronts and rulings.
\end{abstract}

\section{The Hecke algebra}

In this note we make a few observations on the Hecke algebra $\mathscr H_n(z)$. Our main reference is Jones's seminal paper \cite{jones}.
%(no $v$ in the coefficients yet). 
As an algebra, $\mathscr H_n(z)$ is generated by the same symbols $\sigma_1,\ldots,\sigma_{n-1}$ as Artin's braid group $B_n$. In addition to the standard relations of $B_n$ (that is, $\sigma_i\sigma_j=\sigma_j\sigma_i$ for $|i-j|\ge2$ and $\sigma_i\sigma_{i+1}\sigma_i=\sigma_{i+1}\sigma_i\sigma_{i+1}$ for all $i$) we also impose 
\begin{equation}\label{eq:heck}
\sigma_i-\sigma_i^{-1}=z\text{ for all }i.
\end{equation}

For topologists, the Hecke algebra is significant because of its role in the original definition of the Homfly and Jones polynomials \cite{jones} (see also \cite{heck} for a concise survey). Namely, using Ocneanu's trace $\mathrm{Tr}\colon\mathscr H_n(z)\to\Z[z,T]$, the \emph{framed} Homfly polynomial $H$ of a braid $\beta\in B_n$ is obtained as
\begin{equation}\label{eq:homfly}
H_\beta(v,z)=\left(\frac{v^{-1}-v}z\right)^{n-1}\cdot\mathrm{Tr}(\beta)\bigg|_{T=\frac z{1-v^2}}.
\end{equation}
In words: for $\beta\in B_n$, take its natural representation $\beta\in\mathscr H_n(z)$ (a common abuse of notation) and its trace, substitute $T=z/(1-v^2)$ in it, and normalize suitably.

If $\beta$ has exponent sum $w$, then $P_{\widehat\beta}(v,z)=v^w H_\beta(v,z)$ is an expression for the Homfly polynomial of the oriented link $\widehat\beta$ which is the closure of $\beta$. This is the version 
%of the Homfly polynomial 
that satisfies the skein relation
\[v^{-1}P_{\includegraphics[totalheight=8pt]
{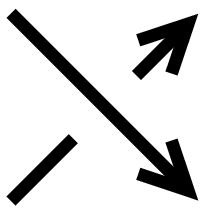}}
-vP_{\includegraphics[totalheight=8pt]{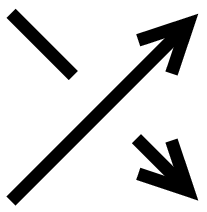}}
=zP_{\includegraphics[totalheight=8pt]
{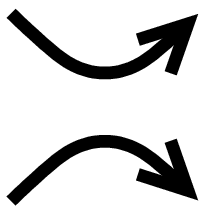}}\]
and the normalizing condition that for the unknot, we have $P_{\includegraphics[totalheight=8pt]{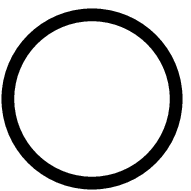}}(v,z)=1$.

From \eqref{eq:homfly}, it is not hard to deduce the famous Morton--Franks--Williams inequalities \cite{morton,fw}:
\begin{equation}\label{ineq:framed}
-n+1\le\text{any exponent of }v\text{ in }H_{\beta}(v,z)\le n-1.
\end{equation}

We will refer to these as the \emph{lower} and \emph{upper MFW estimates}. The well-known bound on the braid index of oriented links is their immediate consequence. We will pay special attention to the ``extremal parts'' of $H_\beta$ with respect to \eqref{ineq:framed}. In that spirit, let the \emph{left column} of $H_\beta$ be denoted with
\[L_\beta(z)=z^{n-1}\text{ times 
the coefficient of }v^{1-n}\text{ in }H_\beta(v,z).\] 
%(If $\beta$ is well chosen, but only then, this is the oriented ruling polynomial of $\widehat\beta$. But here we go in a different direction.)
%The following is going to be our central definition.

\begin{Def}\label{def}
For any $\beta\in B_n$, let $\beta^*$ denote the braid obtained by writing the letters of a braid word representative of $\beta$ in reverse order. Now for any two braids $\alpha$ and $\beta$ on $n$ strands, define their \emph{left inner product} by the formula
\[\langle\alpha,\beta\rangle_L=L_{\alpha\beta^*}(z).\]
\end{Def}

This depends only on the braids and not on any representatives. The left inner product has a well-defined extension to the entire algebra $\mathscr H_n(z)$ as a bilinear form with values in $\Z[z]$. It is also symmetric:

\begin{lemma}
For any $\alpha,\beta\in B_n$, we have $\langle\alpha,\beta\rangle_L=\langle\beta,\alpha\rangle_L$.
\end{lemma}

\begin{proof}
Consider the standard planar diagram for the closure of $\alpha\beta^*$, and a line $l$ in the plane that is (roughly) perpendicular to the braid strands. Rotate the diagram about $l$ in $180^\circ$ in three-space. The result is a standard diagram for the closure of $\beta\alpha^*$ (hence the two closures are isotopic), except that the orientations of the strands are all reversed. Such an overall change of orientation does not affect the Homfly polynomial $P$. Since the exponent sums in $\alpha\beta^*$ and $\beta\alpha^*$ are the same, they also share the same framed Homfly polynomial.
\end{proof}

As a $\Z[z]$--module, $\mathscr H_n(z)$ is free of rank $n!$. We will need two of its well known bases: the set of positive permutation braids, $\mathbf\Omega=\{\,\omega_\pi\,\}_{\pi\in S_n}$, and that of negative permutation braids, $\mathbf N=\{\,\nu_\pi\,\}_{\pi\in S_n}$. 

\section{Legendrian front diagrams}

The proof of our main result in the next section uses some standard techniques from Legendrian knot theory. The summary given here is kept to a bare minimum while the interested reader is referred to \cite{etn}.

The standard contact structure $\xi$ in $\R^3_{xyz}$ is the kernel of the $1$-form $\di z-y\di x$. A smooth link is \emph{Legendrian} if it is everywhere tangent to $\xi$. The \emph{front projection} of such a curve appears on the $xz$-plane. Generic fronts are immersed except for finitely many cusps. Since $y=\di z/\di x$, they do not have self-tangencies or tangents parallel to the $z$-axis (hence they contain an equal number of \emph{left and right cusps}), and at their transverse self-intersections, the strand with lower slope appears as the overcrossing one.

Let $f$ be a generic, oriented front diagram with $2C$ cusps. Its \emph{Thurston--Bennequin number} is $tb(f)=w(f)-C$, where $w$ is the writhe.

An \emph{oriented ruling} \cite{chp, fuchs} of $f$ is a collection $S$ of positive crossings in $f$ with the following properties. After performing the smoothing operation $\includegraphics[scale=.18,viewport=10 10 210 20]{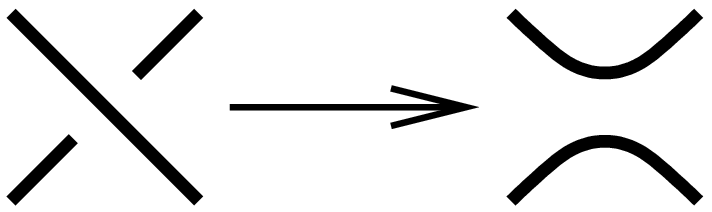}$ to all elements of $S$, the resulting diagram is a (not necessarily disjoint) union of \emph{eyes}, that is simple closed curves so that each contains exactly one left and one right cusp. The two arcs at the site of the smoothing of any $p\in S$ belong to different eyes, say $e_1$ and $e_2$. Finally, if we draw a line parallel to the $z$-axis through $p$, then the segments in which it intersects the interiors of $e_1$ and $e_2$ are either disjoint or contained in one another.

One measure of the complexity of a ruling is $\theta=C-|S|$ and then we may measure the complexity of a front $f$ by summing $z^{1-\theta}$ over all of its oriented rulings. This Laurent polynomial is called the (oriented) \emph{ruling polynomial} of $f$. A key connection was provided by 
Rutherford \cite{dan} when he showed that the ruling polynomial coincides with the coefficient of $v^{tb(f)+1}$ in the Homfly polynomial $P_K(v,z)$ of the oriented link type $K$ represented by $f$. Note in this regard that (by the transverse push-off trick, Bennequin's results on braiding transverse links \cite{benn}, and \eqref{ineq:framed}) the minimum $v$-degree in $P_K(v,z)$ is an upper bound for $tb(f)+1$ for all fronts $f$ representing $K$.

The following construction will be helpful in the next section. Let $\pi\in S_n$ be a permutation. We define a Legendrian tangle $T_\pi$ as follows. Place the permutation matrix 
%(arrangement of rooks on a chessboard) 
of $\pi$ on the $xz$--plane as indicated in the left panel of Figure \ref{fig:egyperm} (with the $1$'s of the matrix represented by dots). Start line segments in the northwest and southwest directions from each dot as shown. Finally, turn the union of the segments into a front diagram as in the right panel, replacing the dots with cusps. Note that $T_\pi$ is a Legendrian realization of the negative permutation braid $\nu_\pi$. 

\begin{figure}
\labellist
\pinlabel $x$ at 390 10
\pinlabel $z$ at 10 380
\endlabellist
   \centering
   \includegraphics[width=3in]{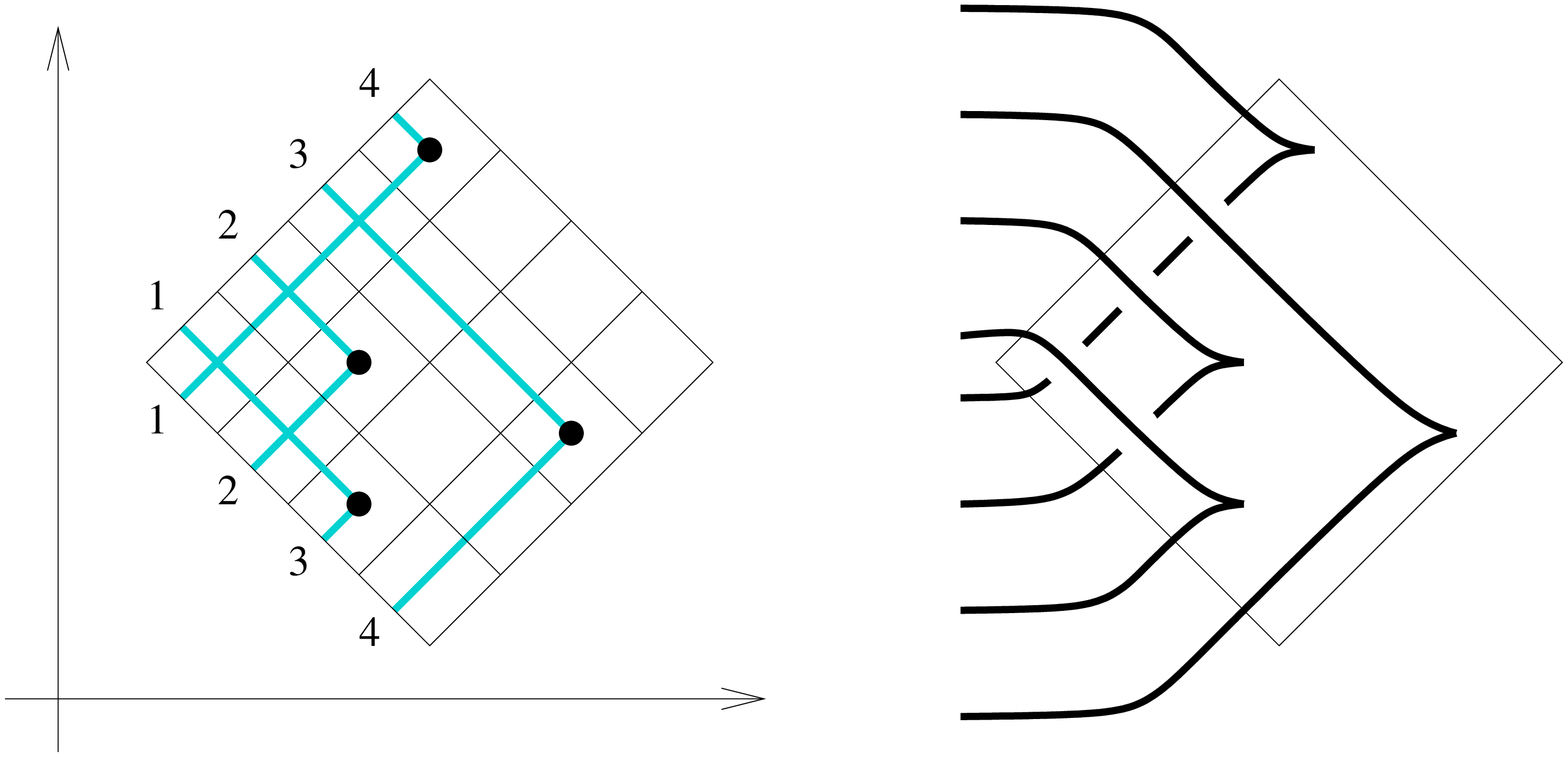} 
   \caption{The Legendrian tangle $T_\pi$ of the three-cycle $\pi=(143)\in S_4$.}
   \label{fig:egyperm}
\end{figure}

\section{The main result}

\begin{all}
The basis $\mathbf N=\{\,\nu_\pi\,\}_{\pi\in S_n}$ of $\mathscr H_n(z)$ consisting of negative permutation braids is orthonormal with respect to the left inner product.
\end{all}

\begin{proof}
Let $\pi$ and $\kappa$ be elements of $S_n$. Consider the Legendrian representative of $\widehat{\nu_\pi\nu^*_\kappa}$ whose front projection is shown in the left side of Figure \ref{fig:ketperm}. It is constructed from the tangle $T_\pi$ and $\varphi(T_\kappa)$, where $\varphi$ is a rotation in $180^\circ$ about the $z$--axis in three-space. Observe that it has oriented rulings if and only if $\pi=\kappa$, and in that case a unique one with $\theta=n$. (Because all crossings are negative in the front, the only possible oriented ruling is the empty set, and for that to work, the left and right cusps have to be readily matched up.) Thus its ruling polynomial is either $\sum z^{1-\theta}=z^{1-n}$ or $0$ depending on whether $\pi=\kappa$ or not. As the Thurston--Bennequin number of this front is $w-n$, where the non-positive quantity $w$ is the combined algebraic number of crossings in $\pi$ and $\kappa$, from Rutherford's theorem we see that this ruling polynomial is the coefficient of $v^{w-n+1}$ in the Homfly polynomial, that is the coefficient of $v^{-n+1}$ in the framed Homfly polynomial. The Proposition follows immediately. 
\begin{figure}
\labellist
\pinlabel $T_{\pi^{-1}}$ at 1460 220
\pinlabel $\beta$ at 1120 100
\endlabellist
   \centering
   \includegraphics[width=4in]{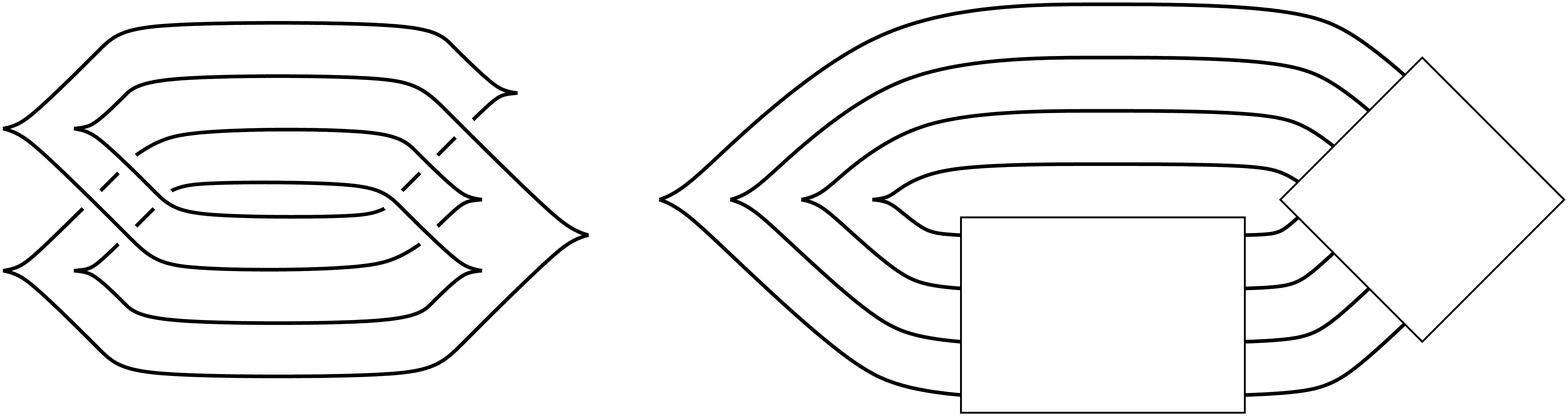} 
   \caption{Front diagram built out of two negative permutation braids (left) and a positive braid and a negative permutation braid (right).}
   \label{fig:ketperm}
\end{figure}
\end{proof}

\begin{kov}
The expansion of any braid $\beta$ in $\mathbf N$ is obtained as
\[\beta=\sum_{\pi\in S_n}\langle\beta,\nu_\pi\rangle_L\cdot\nu_\pi.\]
\end{kov}

In particular if $\beta$ is a positive braid, these coefficients always have a geometric interpretation as ruling polynomials. Namely, $\langle\beta,\nu_\pi\rangle_L$ is the ruling polynomial of the front diagram shown in the right side of Figure \ref{fig:ketperm}. Here, $\beta$ is to be drawn in the indicated box horizontally, with crossings appearing as $\includegraphics[viewport=2 2 20 10]{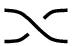}$ . (Here, the key is that the diagram only contains $2n$ cusps. The same is possible for $\beta=\beta_1\nu\beta_2$, where $\beta_1$ and $\beta_2$ are positive and $\nu$ is a negative permutation braid.)

Analogous to Definition \ref{def}, a \emph{right inner product} may be constructed using the coefficient of the other extreme term in the framed Homfly polynomial:
\[\langle\alpha,\beta\rangle_R=(-z)^{n-1}\text{ times the coefficient of }v^{n-1}\text{ in }H_{\alpha\beta^*}(v,z).\]
With respect to this, the basis $\mathbf\Omega$ of positive permutation braids is orthonormal. One easy way of showing this is by noting the identity $\langle\alpha,\beta\rangle_L=\langle\alpha^{-1},\beta^{-1}\rangle_R$ for any braids $\alpha,\beta\in B_n$.

The following theorem appeared as Remark 3.3 in \cite{nyolc}. The present paper grew out of a desire to clarify what was then a rather vague suggestion.

\begin{tetel}
The lower (resp.\ upper) 
MFW estimate \eqref{ineq:framed} is \emph{not} sharp for the braid $\beta$ if and only if $\beta$ is orthogonal in terms of the left (resp.\ right) inner product to the unit braid in the Hecke algebra.
\end{tetel}

We close the paper by indicating two possible directions for future research. The fact that the left and right inner products are polynomial-valued suggests a rich structure that might be further explored. For example, one may ask about specific values of $z$: are some more significant than others? Can we associate special $z$ values to particular braids in a meaningful way?

Also, one may wonder whether other parts of the Homfly polynomial can be used to define inner products, as well as if there are natural inner products that make other bases, such as the Kazhdan--Lusztig basis, orthogonal.

%Acknowledgements.

\end{document}